\documentclass[11pt,reqno,a4paper]{amsart}

\usepackage[latin1]{inputenc}
\usepackage{tikz-cd}
\usepackage{a4wide}
\usepackage{amssymb}
\usepackage{amsfonts}
\usepackage{enumerate}
\usepackage{amsmath}
\usepackage{bbm}
\usepackage{url}

\usepackage{ stmaryrd }
\usepackage{xcolor}
\newtheorem{theorem}{Theorem}

\newtheorem{definition}[theorem]{Definition}

\newtheorem{lemma}[theorem]{Lemma}

\def\eps{\varepsilon}

\def\R{\mathbb{R}}

\def\E{\mathbb{E}}

\DeclareMathOperator{\var}{var}

\title[Shortest distance between observed orbits in distinct Dynamical Systems]{Shortest distance between observed orbits in distinct Dynamical Systems}

\author{Vanessa Barros and Adriana Coutinho}\address{Vanessa Barros\\Departamento de Matemática, Universidade Federal da Bahia\\Av. Ademar de Barros s/n, 40170-110 Salvador, Brazil}\email{vbarrosoliveira@gmail.com}\urladdr{https://vbarrosoliveira.com}\address{Adriana Coutinho\\Département d'Informatique, Institut Supérieur de Technologie (IST)\\Complexe Universitaire de Bikelé-Nzong, B.P. 1378 Libreville, Gabon}\email{adcoutinho.s@gmail.com}

\begin{document}
\maketitle

\begin{abstract}
In this paper, we investigate the asymptotic behavior of the shortest distance between observed orbits in two distinct dynamical systems. 
Given two measure-preserving transformations $(X, T, \mu)$ and $(X, S, \eta)$ and a Lipschitz observation function $f$, we define
\[
\widehat{m}_n^f(x,y) = \min_{i=0,\ldots,n-1} d\big(f(T^i x), f(S^i y)\big).
\]
Under suitable mixing assumptions, we show that the asymptotic rate of decay of $\widehat{m}_n^f(x,y)$ is governed by the symmetric Rényi  divergence of the pushforward measures $f_*\mu$ and $f_*\eta$. 
Our results generalize previous work that consider either a single system or the unobserved case. 
In addition, we discuss the extension of these results to random dynamical systems and illustrate the applicability of the approach with an example.
\end{abstract}

\section{Introduction}

In recent years, the analysis of the behavior of the shortest distance between orbits in dynamical systems has proven to be an important tool for understanding the statistical pro-\\perties of these systems. 
This type of dynamical analysis is related to classical problems in information theory, such as the identification of the largest block of coincidence between two symbolic sequences --- the so-called {Longest Common Substring} problem. 
Motivated by this framework, which includes the study of patterns in DNA sequences and random processes, our work aims to expand the understanding of orbit coincidence by connecting concepts from dynamics and sequence analysis in a more general setting.

The classical study by Arratia and Waterman~\cite{AW} established fundamental results on sequence matching  in the probabilistic setting, motivated by applications in biology and computer science. 
This study, which connects the analysis of recurrent patterns and string matching problems, became a reference for subsequent formulations in the dynamical systems framework.

Following this line, in~\cite{BLR}, a generalization of the {Longest Common Substring} 
problem to the dynamical systems context was proposed through the study of the behavior 
of the shortest distance between two orbits. In this setting, the presence of long common substrings in symbolic sequences 
can be reinterpreted in terms of the approximation between trajectories generated 
by a transformation.

{Results on matching blocks and minimal distance functions have also been obtained for continued fractions \cite{ShiZhou}, conformal iterated function systems \cite{Shi} and more general iterated function systems including classical examples such as Lüroth, $N$-ary expansions, and the Cantor system \cite{ShiTanZhou}. Additionally, the statistical distribution of the closest encounter between observations along different trajectories has also been analyzed in rapidly mixing systems, revealing Gumbel-type laws and connections with extremal indices and matching blocks \cite{Caby}.}
Recent algorithmic developments have further extended this classical formulation by allowing gaps and introducing connections with co-linear chaining techniques in computational biology, as discussed in \cite{banerjee}. 
{From an algorithmic point of view, dynamic variants of the Longest Common Substring problem have been explored, for example in \cite{Amir}. Minimal-distance phenomena have also been investigated in measure-preserving systems, where almost sure orbit closeness under mixing assumptions has been established \cite{Kirsebom}.}

In~\cite{coutinho}, the authors analyze the asymptotic behavior of the shortest distance between observed orbits in a single dynamical system, assuming suitable mixing properties for H\"older observables. 
In contrast, the present work extends this analysis to the case of two distinct systems, requiring the corresponding condition to hold only for Lipschitz observables.  
Our assumption is therefore weaker in terms of regularity, leading to results that hold under less restrictive hypotheses. 
{We also note that the mixing framework considered in~\cite{coutinho} and
in~\cite{BLR} requires an exponential decay of correlations,} given by $\theta_n = a^{\, n}$ with 
$0 \leq a < 1$. In the present work, however, we allow a {stretched exponential} 
rate of the form
\[
\theta_n = a^{\, n^{\alpha}}, \qquad 0 \leq a < 1,\; 0 < \alpha < 1,
\]
which is  weaker and still provides sufficient control over the dependence 
between the observed orbits.

{Compared with \cite{BLR}, where the shortest distance problem is studied within a single
dynamical system and its asymptotic behavior is governed by the correlation dimension
of the invariant measure, here we deal with two distinct measure-preserving dynamical
systems and introduce an observation function $f$.
This leads naturally to a characterization in terms of the symmetric Rényi divergence
between the pushforward measures $f_*\mu$ and $f_*\eta$. It is important to notice that working with an observation $f$ allows us to study the shortest distance between two distinct random dynamical systems.
This shows the present framework goes beyond that of \cite{BLR}, not only in the
nature of the systems and the limiting geometric quantity, but also in the level of
dependence required between observations.}


Later, in~\cite{BarrosRousseau}, the analysis of the shortest distance problem was extended to the setting of multiple orbits and linked to generalized fractal dimensions, providing a broader dynamical framework that also encompasses random systems and symbolic sequences. 

{In the random setting, \cite{Rousseau21} studied the quenched behavior of the longest common substring and its relation to the Rényi entropy.}
Continuing this line of investigation,~\cite{RT} analyze the
shortest distance between orbits in dynamical systems with slow
mixing properties. They show that the decay of this minimal distance still follows
a polynomial law in $n$ for several slowly mixing systems, thereby extending and
refining the results of~\cite{BLR}. A key ingredient in their approach is the use of
inducing schemes, which allow one to transfer quantitative estimates obtained for
the induced system to the original one. This method yields new classes of examples
where the scaling of orbit proximity can be rigorously established, even in settings
where no standard mixing estimates are available.

Subsequently,~\cite{GRS} investigated the shortest distance between orbits in random dynamical systems with strong mixing properties. 
While earlier work focused on deterministic dynamics or on the {annealed} formulation of the problem, these authors analyzed the more delicate {quenched} case and showed that the  asymptotic behavior is more intricate, involving two distinct correlation dimensions and exhibiting a non-smooth dependence of the asymptotic exponent.

Recently, in~\cite{ShiLiuZhang} and~\cite{Kirsebom}, the authors analyzed the asymptotic behavior of the shortest distance between orbits in different dynamical systems, defined by $(X, T, \mu)$ and $(X, S, \eta)$, using a compatible metric $d$.

They showed that, under certain mixing conditions, the quantity
\[
\widehat{m}_n(x,y) = \min_{i=0,\ldots,n-1} d\big(T^i x, S^i y\big),
\]
converges, and that its limit is related to the symmetric Rényi divergence of the measures $\mu$ and $\eta$.
Motivated by~\cite{ShiLiuZhang} and~\cite{coutinho}, we extend this framework by introducing a measurable observation function $f$ and studying the asymptotic behavior of the shortest distance between {observed orbits} in different dynamical systems, defined as
\[
\widehat{m}_n^f(x,y) = \min_{i=0,\ldots,n-1} d\big(f(T^i x), f(S^i y)\big).
\]

It is worth noting that, in the framework of~\cite{coutinho} { as in~\cite{BLR}}, the minimal 
distance is computed by allowing the two trajectories to be compared at different time 
instants, that is, $i \neq j$.
In the present work, by contrast, the comparison is restricted to matching times. 
Because we deal simultaneously with two different dynamical systems and an observation 
function~$f$, these ingredients already introduce substantial technical difficulties, and 
the available tools permit us to treat only the synchronous case.

In this way, we investigate how the asymptotic behavior of the shortest distance between observed orbits is related to the symmetric Rényi divergence of the pushforward measures $f_* \mu$ and $f_* \eta$. 
That is, if the symmetric Rényi divergence  $C^f_{\mu,\eta}$ exists, and under mixing conditions for both systems $(X,\mu,T)$ and $(X,\eta,S)$, we deduce that for $\mu \otimes \eta$-almost every  $(x,y) \in X \times X$,
\[
\lim_{n \to +\infty} \frac{\log \widehat{m}_n^f(x,y)}{-\log n} = \frac{1}{C^f_{\mu,\eta}}.
\]
If $f$ is the identity function, we extend the study of orbit coincidence to the case of two distinct dynamical systems, as presented in~\cite{ShiLiuZhang}, under mixing conditions and for Lipschitz functions. 
This class of observables is particularly relevant for meaningful applications.

The use of the symmetric Rényi divergence in shortest-distance problems  was introduced in~\cite{KO} and later formalized in~\cite{abadi-lambert}
as a refinement of earlier notions based on correlation dimension. 
While the correlation dimension captures averaged geometric overlap and is
commonly used to characterize the fractal structure of an underlying attractor,
it does not directly describe the decay of matching probabilities between
observations. In contrast, the symmetric Rényi divergence provides an
information-theoretic characterization of the exponential rate at which such
matching probabilities decrease, a distinction that becomes particularly
relevant when shortest-distance problems involve different observations or
invariant measures.

The paper is organized as follows. 
In Section~\ref{secshort}, we introduce the main definitions and state our principal results. 
Section~\ref{sec-proof} is devoted to the proofs and technical lemmas. 
Finally, in Section~\ref{sectionshortestrandom}, we extend our results to random dynamical systems and provide an example that illustrates their practical relevance.

\medskip

\section{Shortest distance between orbits}\label{secshort} 


Let \((X, d)\) be a finite-dimensional metric space and \(\mathcal{A}\) its Borel \(\sigma\)-algebra.  
Let \((X, \mathcal{A}, \mu, T)\) and \((X, \mathcal{A}, \eta, S)\) be two measure-preserving dynamical systems, and let \(f: X \to Y \subset \R^n\) be a measurable function, called the observation.   
We aim to study the behavior of the shortest distance between two different observed orbits:  
\[
\widehat{m}_n^f(x, y) = \min_{i = 0, \dots, n-1}  d(f(T^i x), f(S^i y)).
\]  
We will show that the asymptotic behavior of \(\widehat{m}_n^f\) as \(n \to \infty\) is related to the symmetric Rényi divergence  (defined below) of the pushforward measures \(f_*\mu\) and \(f_*\eta\). Recall that the pushforward measure \(f_*\mu\) is given by
\(f_*\mu(\cdot) := \mu(f^{-1}(\cdot))\).

\begin{definition}
Let \((X, \mathcal{A}, \mu, T)\) and \((X, \mathcal{A}, \eta, S)\) be measure-preserving systems. We define the lower and upper symmetric Rényi divergence of the measures \(\mu\) and \(\eta\) as
\[
\underline{C}_{\mu, \eta} = \liminf_{r \to 0} \frac{\log \int_X \mu(B(x, r)) \, d\eta(x)}{\log r}
\quad \text{and} \quad
\overline{C}_{\mu, \eta} = \limsup_{r \to 0} \frac{\log \int_X \mu(B(x, r)) \, d\eta(x)}{\log r}.
\]

\end{definition}

For simplicity, we denote the lower and upper symmetric Rényi divergence of the pushforward measures \(f_*\mu\) and \(f_*\eta\) by
\[\underline{C}^f_{\mu, \eta} =\underline{C}_{f_*\mu, f_*\eta} \quad \text{and} \quad \overline{C}^f_{\mu, \eta} =\overline{C}_{f_*\mu, f_*\eta}.\]
When the \(\liminf\) and \(\limsup\) coincide, we simply write \(C^f_{\mu, \eta}\) for the common value. Observe that
\(
C_{\mu,\eta} = C_{\eta,\mu},
\)
which motivates the terminology {symmetric}.

Our first result shows that  for sufficiently large $n,\ \widehat{m}^f$ cannot decay faster than \(n^{-1/\underline{C}^f_{\mu, \eta}}\):

\begin{theorem}\label{thineq}
Let \((X, \mathcal{A}, \mu, T)\) and \((X, \mathcal{A}, \eta, S)\) be two measure-preserving systems. Consider an observation \(f : X \to Y\) such that \(\underline{C}^f_{\mu, \eta} > 0\). Then, for \(\mu \otimes \eta\)-almost every \((x, y) \in X \times X\), we have
\[
\underset{n \to +\infty}{\overline\lim} \frac{\log \widehat{m}_n^f(x, y)}{-\log n} \leq \frac{1}{\underline{C}^f_{\mu, \eta}}.
\]
\end{theorem}


\medskip

Proving the equality is more delicate and less straightforward.
Obtaining a matching lower bound requires two additional assumptions,
which are not needed for the upper bound.

The first one, {(H1)}, is a mixing-type condition formulated in terms of decay of correlations for Lipschitz observables under the dynamics of both systems \((X, \mu, T)\) and \((X, \eta, S)\). It ensures that statistical dependencies between observations decay sufficiently fast in time:

(H1) There exist constants \(0 \le a < 1\) and \(\alpha > 0\) such that,
for all \( \psi, \phi \in \mathrm{Lip}(X,\mathbb{R}) \) and all \( n \in \mathbb{N}^* \),
we have
\[
\left| 
\int_X \psi(x)\,\phi(T^n x)\, d\mu 
- 
\int_X \psi\, d\mu \int_X \phi\, d\mu 
\right|
\leq 
\|\psi\|_{\mathrm{Lip}} \, \|\phi\|_{\mathrm{Lip}} \, \theta_n,
\]
and
\[
\left| 
\int_X \psi(x)\,\phi(S^n x)\, d\eta 
- 
\int_X \psi\, d\eta \int_X \phi\, d\eta 
\right|
\leq 
\|\psi\|_{\mathrm{Lip}} \, \|\phi\|_{\mathrm{Lip}} \, \theta_n,
\]
where $
\theta_n = a^{n^\alpha}.
$

The second one, {(H2)}, is a regularity condition on the pushforward measure \(f_*\mu\), which controls the measure of annuli in terms of their radii and thicknesses:

(H2) There exist $r_0>0$, $\xi\geq0$ and $\beta\geq 1$ such that for $f_*\mu$-almost every $y\in Y$ and any $r_0>r>\rho>0$,
\[f_*\mu(B(y,r+\rho)\backslash B(y,r-\rho))\leq r^{-\xi}\rho^\beta.\]
\noindent
Since the problem under consideration is symmetric in the two systems
$(X,\mu,T)$ and $(Y,\eta,S)$, it suffices to require the annuli regularity condition for only one of the measures.

\smallskip

Now we can state our main result.

\begin{theorem}\label{thprinc}
Let $(X,\mathcal{A},\mu,T)$ and $(X,\mathcal{A},\eta,S)$ be two measure-preserving systems satisfying \textnormal{(H1)} and \textnormal{(H2)}. Suppose that $f$ is a Lipschitz observation such that ${C}^f_{\mu,\eta}$ exists and is strictly positive. Then, for $\mu\otimes\eta$-almost every $(x,y) \in X \times X$, we have
\[
\lim_{n \to +\infty} \frac{\log \widehat{m}_n^f(x,y)}{-\log n} = \frac{1}{{C}^f_{\mu,\eta}}.
\]
\end{theorem}



\medskip

\section{Proofs of the main results}\label{sec-proof}

\begin{proof}[Proof of Theorem \ref{thineq}]

For $\epsilon > 0$, define  
\[
k_n = \frac{\log n + \log \log n}{\underline{C}^f_{\mu,\eta} - \epsilon}, \quad \text{and} \quad r_n = e^{-k_n}.
\]  
For each $i = 0, \ldots, n-1$, set
\begin{equation*}
A^f_{i}(y) = T^{-i}\big[f^{-1}(B(f(S^i y), r_n))\big],
\end{equation*}
and define the counting function
\begin{equation*}
Q^f_n(x, y) = \sum_{i=0}^{n-1} \mathbbm{1}_{A^f_i(y)}(x).
\end{equation*}
By construction,  
 
\[
Q^f_n(x, y) \geq 1 \Leftrightarrow  d\big(f(T^i x), f(S^i y)\big) < r_n \text{ for some } i.
\]
Since \(\widehat{m}^f_n(x, y)\) is  the minimum of these distances, we get
\begin{equation*} \label{eq:MnQn}
\{ (x, y) : \widehat{m}^f_n(x, y) < r_n\} = \{  (x, y) : Q^f_n(x, y) \geq 1\}.
\end{equation*}
Using Markov's inequality, we obtain  
\begin{equation} \label{eq:MnQn1}
\mu \otimes \eta \big((x, y) : \widehat{m}^f_n(x, y) < r_n\big)  
\leq \mathbb{E}(Q^f_n).
\end{equation}
Let
\[
B_n = \{(x, y) : \widehat{m}^f_n(x, y) < r_n\}.
\]  
Our goal is to show that there exists a suitable subsequence $(n_\ell)_\ell$ for which
\begin{equation} \label{B_n}
\sum_\ell \mu \otimes \eta(B_{n_\ell}) < +\infty.
\end{equation}
We now fix
\[
n_\ell = \lceil e^{\ell^2} \rceil, \qquad \ell \in \mathbb{N},
\]
and we shall verify later that this choice indeed satisfies \eqref{B_n}. 
Assuming for the moment that this is true,
the Borel--Cantelli Lemma implies that for $\mu \otimes \eta$-almost every $(x,y)\in X\times X$, 
for all $\ell$ sufficiently large,
\[
\widehat{m}^f_{n_\ell}(x,y)\ge r_{n_\ell},
\]
which is equivalent to
\begin{equation}\label{des1}
\frac{\log \widehat{m}^f_{n_\ell}(x,y)}{-\log n_\ell}
\le 
\frac{1}{\underline{C}^f_{\mu,\eta}-\epsilon}.
\end{equation}
Hence,
\[
\overline{\lim_{\ell\to\infty}}
\frac{\log \widehat{m}^f_{n_\ell}(x,y)}{-\log n_\ell}
\le 
\frac{1}{\underline{C}^f_{\mu,\eta}-\epsilon}.
\]

To extend the last estimate  to the whole sequence, note that for every $n$ there exists $\ell$ such that 
$e^\ell \le n \le e^{\ell+1}$, and since $(\widehat{m}^f_n)_n$ is decreasing,
\[
\frac{\log \widehat{m}^f_{n_\ell}(x,y)}{-\log n_{\ell+1}}
\le 
\frac{\log \widehat{m}^f_n(x,y)}{-\log n}
\le 
\frac{\log \widehat{m}^f_{n_{\ell+1}}(x,y)}{-\log n_\ell}.
\]
Taking upper limits, using 
$\displaystyle\lim_{\ell\to\infty}\frac{\log n_\ell}{\log n_{\ell+1}}=1$ and inequality \eqref{des1}
we get
\[
\overline{\lim_{n\to\infty}}
\frac{\log \widehat{m}^f_n(x,y)}{-\log n}
=
\overline{\lim_{\ell\to\infty}}
\frac{\log \widehat{m}^f_{n_\ell}(x,y)}{-\log n_\ell}
\le 
\frac{1}{\underline{C}^f_{\mu,\eta}-\epsilon}.
\]
and the theorem is proved since $\eps$ is arbitrary.

It remains to prove inequality  \eqref{B_n} for  \(n_\ell = \lceil e^{\ell^2} \rceil \). By inequality \eqref{eq:MnQn1}, it suffices to estimate \( \mathbb{E}(Q^f_n) \). Using  Fubini's theorem and the  invariance of $\mu$ under $T$ and $\eta$ under $S$, we conclude that
    \begin{eqnarray*}
\E(Q^f_n)&=&\iint
\sum_{i=0}^{n-1}
\mathbbm{1}_{A^f_{i}(y)}(x)d\mu\otimes\eta(x,y)\nonumber \\
&=&\sum_{i=0}^{n-1}\int \mu\left(T^{-i}f^{-1}B(f(S^iy),r_n)\right)d\eta(y)\nonumber \\
&=&n\int f_{\ast}\mu\left(B(z,r_n)\right)d(f_{\ast}\eta)(z).
\end{eqnarray*}
From the definitions of the lower symmetric Rényi divergence and of \(k_n\), 
we obtain for large \(n\) that 
\[
\mu \otimes \eta \left(B_n\right) \leq n r_n^{\underline{C}^f_{\mu, \eta} - \epsilon} = \frac{1}{\log n}.
\]
Finally, we replace \(n_\ell = \lceil e^{\ell^2} \rceil\), and get
\[
\mu \otimes \eta \left(B_{n_\ell}\right) \leq \frac{1}{\log n_\ell} \leq \frac{1}{\ell^2},
\]
which proves \eqref{B_n}.

\end{proof}


We are now ready to prove Theorem~\ref{thprinc}.  
We will need the following result, which controls the push-forward measure of balls in the space of Lipschitz-continuous functions. The proof follows the same strategy as in~\cite{coutinho}.   

While the structure of the argument is similar to the original one, the statement is not identical. 
Our version is formulated in terms of Lipschitz continuity rather than H\"older regularity.

\begin{lemma}\label{HA}
Let \( (X, \mathcal{A}, \mu, T) \) be a dynamical system with observation \( f \).  
If it satisfies condition~\((H2)\), then there exist constants \( 0 < r_0 < 1 \), \( c \geq 0 \), and \( \zeta \geq 0 \) such that for any \( 0 < r < r_0 \), the function
\[
\psi_1: x \mapsto f_*\mu \big( B(f(x), r) \big)
\]
belongs to \( \mathrm{Lip}(X, \mathbb{R}) \), and
\[
\| \psi_1 \|_{\mathrm{Lip}} \leq c r^{-\zeta} \ .
\]
\end{lemma}

\begin{proof}[Proof of Theorem \ref{thprinc}]

First, note that it is sufficient to prove the following inequality:
\[
\underset{n \rightarrow +\infty}{\underline\lim} \frac{\log \widehat{m}^f_n(x, y)}{-\log n} \geq \frac{1}{\overline{C}_{\mu, \eta}^f}.
\]

For \(\epsilon > 0\), define
\[
k_n = \frac{1}{\overline{C}^f_{\mu,\eta} + \epsilon}(\log n + b \log \log n) \quad \text{and} \quad r_n = e^{-k_n},
\]
where \(b\) will be chosen appropriately later in the proof.

Our goal now is to estimate
\begin{equation*}
\mu \otimes \eta \left( (x, y) : \widehat{m}^f_n(x, y) \geq r_n \right),
\end{equation*}
and then apply the Borel-Cantelli Lemma, as we did in the proof of Theorem~\ref{thineq}.
Proceeding similarly to that proof, and applying Chebyshev's inequality, we obtain
\begin{align}
&\mu \otimes \eta \left((x, y) : \widehat{m}^f_n(x, y)\geq r_n\right) \leq\mu\otimes\eta \left((x,y):Q^f_n(x,y)=0\right) \nonumber \\
\leq&\mu\otimes\eta \left((x,y):|Q^f_n(x,y)-\E(Q^f_n)|\geq |\E(Q^f_n)|\right)
\leq\frac{\var(Q^f_n)}{\E(Q^f_n)^2}, \label{eqmnvarsn}
\end{align}
where \(Q^f_n\) is given by
\begin{equation*}
Q^f_n(x, y) = \sum_{i=0}^{n-1} \mathbbm{1}_{A^f_i(y)}(x),
\end{equation*} 
and 
\begin{equation*}
A^f_{i}(y) = T^{-i}\big[f^{-1}(B(f(S^i y), r_n))\big].
\end{equation*}

\medskip

Thus, we need to control the variance of $Q^f_n$. First of all, we have
\begin{align}\label{varestimate}
\var(Q^f_n)
=&\sum_{i,i'=0}^{n-1}\mathrm{Cov}(\mathbbm{1}_{A^f_{i}},\mathbbm{1}_{A^f_{i'}})=\sum_{i,i'=0}^{n-1}\left(\iint \mathbbm{1}_{A^f_{i}}\mathbbm{1}_{A^f_{i'}}-\iint\mathbbm{1}_{A^f_{i}}\iint\mathbbm{1}_{A^f_{i'}}\right)\\
=&\sum_{i,i'=0}^{n-1}\iint \mathbbm{1}_{f^{-1}B(f(S^iy),r_n)}(T^ix)\mathbbm{1}_{f^{-1}B(f(S^{i'}y),r_n)}(T^{i'}x)\, d\mu\otimes\eta(x,y)\nonumber\\
&-\E(Q^f_n)^2,\nonumber
\end{align}
since
$
\displaystyle\sum_{i=0}^{n-1}\iint \mathbbm{1}_{A^f_i}(x,y) \, d\mu\otimes\eta(x,y) = \mathbb{E}(Q^f_n).
$

\medskip
We next focus on estimating the double sum in~\eqref{varestimate}. Because the indicator function is not Lipschitz continuous, our hypothesis (H1) does not apply directly. However, we can approximate the indicator function by a sequence of Lipschitz functions. We now proceed to construct such a sequence (see, e.g.,~\cite{rapid}). 

\medskip
Let \(\eta_{r_n} : [0, \infty) \to \mathbb{R}\) be defined by:
\[
\eta_{r_n}(t) = 
\begin{cases}
1, & \text{if } 0 \leq t \leq r_n, \\
1 - \dfrac{t - r_n}{\rho r_n}, & \text{if } r_n < t \leq (1 + \rho)r_n, \\
0, & \text{if } t > (1 + \rho)r_n,
\end{cases}
\]
where \(\rho > 0\) will be defined later. 
\medskip

Note that \(\eta_{r_n}\) is a \(\frac{1}{\rho r_n}\)-Lipschitz function, and satisfies the inequality:
\begin{equation}\label{desigualdade-eta}
\mathbbm{1}_{[0, r_n]} \leq \eta_{r_n} \leq \mathbbm{1}_{[0, (1+\rho)r_n]}.
\end{equation}

Given \( z \in Y \), define \( \varphi_{z, r_n} : Y \to \mathbb{R} \) by
$
\varphi_{z, r_n}(x) = \eta_{r_n}(d(z, x)).$
Since \( f \) is Lipschitz, it follows that \( \varphi_{f(y), r_n} \circ f \) is 
\[
\frac{L}{\rho r_n}\text{-Lipschitz}
\]
for any \( x \in X \), where \( L \) denotes the minimal Lipschitz constant of \( f \).

Moreover from inequality \eqref{desigualdade-eta}, we have

\begin{eqnarray}\label{eqapprox}
\mathbbm{1}_{f^{-1}B(f(S^iy),r_n)}(x) =\mathbbm{1}_{[0,r_n]}(d(f(S^iy),f(x)))
 \leq  \varphi_{f(S^iy),r_n}(f(x)).
\end{eqnarray}

We will consider two different cases. Let us fix \( g = g(n) := \log(n^\gamma) \) for some \( \gamma > 0 \) to be chosen later.

\medskip

\noindent\textbf{Case 1:} \( |i - i'| > g \). From the mixing property and inequality~\eqref{eqapprox}, we obtain:
\begin{align*}
&\iint \mathbbm{1}_{f^{-1}B(f(S^i y), r_n)}(T^i x) \cdot \mathbbm{1}_{f^{-1}B(f(S^{i'} y), r_n)}(T^{i'} x) \, d(\mu \otimes \eta)(x,y) \\
&= \iint \mathbbm{1}_{f^{-1}B(f(S^i y), r_n)}(T^{i-i'} x) \cdot \mathbbm{1}_{f^{-1}B(f(S^{i'} y), r_n)}(x) \, d(\mu \otimes \eta)(x,y) \\
&\leq \iint \varphi_{f(S^i y), r_n}(f(T^{i - i'} x)) \cdot \varphi_{f(S^{i'} y), r_n}(f(x)) \, d(\mu \otimes \eta)(x,y).
\end{align*}

Using Hypothesis (H1) and inequality~\eqref{desigualdade-eta} we obtain:
\begin{align*}
&\iint \mathbbm{1}_{f^{-1}B(f(S^i y), r_n)}(T^i x) \cdot \mathbbm{1}_{f^{-1}B(f(S^{i'} y), r_n)}(T^{i'} x) \, d(\mu \otimes \eta)(x,y) \\
&\leq \int \left( \int \varphi_{f(S^i y), r_n}(f(T^{i - i'} x)) \, d\mu(x) \cdot \int \varphi_{f(S^{i'} y), r_n}(f(x)) \, d\mu(x) \right) d\eta(y) \\
&\quad + \theta_g \cdot \|\varphi_{f(S^i y), r_n}\circ f\|_{\mathrm{Lip}} \cdot \|\varphi_{f(S^{i'} y), r_n}\circ f\|_{\mathrm{Lip}} \\
&\leq \int \left( \int \mathbbm{1}_{[0, (1 + \rho) r_n]}(d(f(S^i y), f(T^{i - i'} x))) \, d\mu(x) \cdot \int \mathbbm{1}_{[0, (1 + \rho) r_n]}(d(f(S^{i'} y), f(x))) \, d\mu(x) \right) d\eta(y) \\
&\quad + \frac{{L}^2}{\rho^2 r_n^2}  \theta_g.
\end{align*}

Note that here (H1) was applied with $\theta_{|i-i'|}$, which is bounded by $\theta_g$. This type of replacement will appear again later in the proof.\par\medskip

Moreover, using the mixing property together with~\eqref{eqapprox}, we obtain:

\begin{align}
&\iint \mathbbm{1}_{f^{-1}B(f(S^i y), r_n)}(T^i x) \cdot \mathbbm{1}_{f^{-1}B(f(S^{i'} y), r_n)}(T^{i'} x) \, d(\mu \otimes \eta)(x,y) \nonumber \\
&\leq \int \left( \int \mathbbm{1}_{f^{-1}B(f(S^i y), (1+\rho) r_n)}(x) \, d\mu(x) \cdot \int \mathbbm{1}_{f^{-1}B(f(S^{i'} y), (1+\rho) r_n)}(x) \, d\mu(x) \right) d\eta(y) \nonumber \\
&\quad + \frac{L^2}{\rho^2 r_n^2}  \theta_g \nonumber \\
&= \int f_*\mu\big(B(f(S^i y), (1+\rho) r_n)\big) \cdot f_*\mu\big(B(f(S^{i'} y), (1+\rho) r_n)\big) \, d\eta(y) + \frac{L^2}{\rho^2 r_n^2}  \theta_g. \label{eqmix1}
\end{align}

To estimate the term
\[
\int f_*\mu\big(B(f(S^i y), (1+\rho) r_n)\big) \cdot f_*\mu\big(B(f(S^{i'} y), (1+\rho) r_n)\big) \, d\eta(y),
\]

we begin by considering the difference
\begin{eqnarray*}
&&\int f_*\mu\big(B(f(S^i y), (1+\rho) r_n)\big) \cdot f_*\mu\big(B(f(S^{i'} y), (1+\rho) r_n)\big) \, d\eta(y) \\
&-& \int f_*\mu\big(B(f(S^i y), r_n)\big) \cdot f_*\mu\big(B(f(S^{i'} y), r_n)\big) \, d\eta(y),
\end{eqnarray*}
and then add and subtract an intermediate term to control the measure of the corresponding annuli:
\begin{eqnarray*}
&&\int f_*\mu\big(B(f(S^i y), (1+\rho) r_n)\big) \cdot f_*\mu\big(B(f(S^{i'} y), (1+\rho) r_n)\big) \, d\eta(y) \\
&-& \int f_*\mu\big(B(f(S^i y), r_n)\big) \cdot f_*\mu\big(B(f(S^{i'} y), r_n)\big) \, d\eta(y) \\
&\leq& \int f_*\mu\big(B(f(S^i y), (1+\rho) r_n)\big) \cdot \left[ f_*\mu\big(B(f(S^{i'} y), (1+\rho) r_n)\big) - f_*\mu\big(B(f(S^{i'} y), r_n)\big) \right] \, d\eta(y) \\
&+& \int f_*\mu\big(B(f(S^{i'} y), r_n)\big) \cdot \left[ f_*\mu\big(B(f(S^i y), (1+\rho) r_n)\big) - f_*\mu\big(B(f(S^i y), r_n)\big) \right] \, d\eta(y).
\end{eqnarray*}

 Note that, for every $z \in Y$, we have
\begin{align*}
f_*\mu\big(B(z,(1+\rho) r_n)\big) - f_*\mu\big(B(z,r_n)\big)
\leq f_*\mu\Big(B(z, r_n + \rho) \setminus B(z, r_n - \rho)\Big).
\end{align*}

Hence, by hypothesis~(H2) we obtain:

\begin{eqnarray*}
&&\int f_*\mu\big(B(f(S^i y), (1+\rho) r_n)\big) \cdot f_*\mu\big(B(f(S^{i'} y), (1+\rho) r_n)\big) \, d\eta(y) \\
&-& \int f_*\mu\big(B(f(S^i y), r_n)\big) \cdot f_*\mu\big(B(f(S^{i'} y), r_n)\big) \, d\eta(y) \\
&\leq& r_n^{-\xi} \rho^{\beta} \int \left[ f_*\mu\big(B(f(S^i y), (1+\rho) r_n)\big) + f_*\mu\big(B(f(S^{i'} y), r_n)\big) \right] \, d\eta(y).
\end{eqnarray*}

Using the invariance of \( \eta \) under \( S\), if \( \rho \) is sufficiently small we have

\begin{eqnarray}\label{des2}
&&\int f_*\mu\big(B(f(S^i y), (1+\rho) r_n)\big) \cdot f_*\mu\big(B(f(S^{i'} y), (1+\rho) r_n)\big) \, d\eta(y) \nonumber\\
&-& \int f_*\mu\big(B(f(S^i y), r_n)\big) \cdot f_*\mu\big(B(f(S^{i'} y), r_n)\big) \, d\eta(y) \nonumber\\
&\leq&2r^{-\xi}_n\rho^{\beta} \int f_*\mu(B(f(y),2r_n)) \ d\eta(y)
\end{eqnarray}

Therefore, combining inequalities~\eqref{eqmix1} and~\eqref{des2}, we deduce
\begin{eqnarray}\label{des3}
&& \iint \mathbbm{1}_{f^{-1}B(f(S^i y), r_n)}(T^i x)\, 
      \mathbbm{1}_{f^{-1}B(f(S^{i'} y), r_n)}(T^{i'} x)\, 
      d(\mu \otimes \eta)(x,y)  \nonumber \\[0.2cm]
&\leq& \frac{L^2}{\rho^2 r_n^2} \, \theta_g 
\;+\; 2 r_n^{-\xi} \rho^{\beta} 
      \int f_* \mu\!\left( B(f(y), 2r_n) \right) \, d\eta(y) \\[0.2cm]   
&&\;+\; \int f_* \mu\!\left( B(f(S^i y), r_n) \right)\,
         f_* \mu\!\left( B(f(S^{i'} y), r_n) \right)\, d\eta(y) \nonumber
\end{eqnarray}

We now estimate
\[
\int f_* \mu\big( B(f(S^i y), r_n) \big)\,
     f_* \mu\big( B(f(S^{i'} y), r_n) \big)\, d\eta(y).
\]
Let \( \psi(y) = f_* \mu\big( B(f(y), r_n) \big) \). 
Using the $S$-invariance of \( \eta \) and applying the mixing hypothesis~(H1) to the function \( \psi \), we obtain
\[
\int \psi(S^i y)\, \psi(S^{i'} y)\, d\eta(y)
\;\leq\;
\left( \int \psi(y)\, d\eta(y) \right)^2
\;+\;
\theta_g\, \|\psi\|_{\mathrm{Lip}}^{2}.
\]

Applying Lemma~\ref{HA} and returning to the original expression in
inequality~\eqref{des3}, we infer
\begin{eqnarray*}
&& \iint \mathbbm{1}_{f^{-1}B(f(S^i y), r_n)}(f(T^i x)) \, \mathbbm{1}_{f^{-1}B(f(S^{i'} y), r_n)}(f(T^{i'} x)) \, d(\mu \otimes \eta)(x, y) \\
&\leq& \dfrac{L^2}{\rho^2 r_n^2} \, \theta_g 
+ 2 r_n^{-\xi} \rho^{\beta} \int f_* \mu\big( B(f( y), 2r_n) \big) \, d\eta(y) \\
&& + c^2 \theta_g \, r_n^{-2\zeta} 
+ \left( \int f_* \mu\big( B(f(y), r_n) \big) \, d\eta(y) \right)^2.
\end{eqnarray*}

We now set this estimate aside and proceed with the analysis of \(\var(Q^f_n)\) in Case~2.

\medskip

\noindent\textbf{Case 2:} \( |i - i'| \leq g \).  
For this case, the boundedness of the indicator function and the invariance of \( \mu \) under \( T \) and of \( \eta \) under \( S \) yield:
\begin{eqnarray*}
& &\iint \mathbbm{1}_{f^{-1}B(f(S^i y),r_n)}(T^i x)\mathbbm{1}_{f^{-1}B(f(S^{i'} y),r_n)}(T^{i'} x)\, d(\mu \otimes \eta)(x,y)\\
&\leq &\iint \mathbbm{1}_{f^{-1}B(f(S^i y),r_n)}(T^i x)\, d(\mu \otimes \eta)(x,y)\\
&= &\int f_*\mu(B(f(S^i y),r_n)) \, d\eta(y)= \int f_*\mu(B(f(y),r_n)) \, d\eta(y).
\end{eqnarray*}

We are now ready to estimate the double sum in~\eqref{varestimate}.  
Combining Cases~1 and~2, we arrive at

\begin{align}\label{duas-somas}
&\sum_{i,i'=0}^{n-1}\iint \mathbbm{1}_{f^{-1}B(f(S^i y),r_n)}(T^i x)\mathbbm{1}_{f^{-1}B(f(S^{i'} y),r_n)}(T^{i'} x)\, d(\mu \otimes \eta)(x,y) \\
&\leq n^2 L^2 \rho^{-2} r_n^{-2} \theta_g 
+ 2n^2 r_n^{-\xi} \rho^{\beta} \int f_*\mu(B(f(y), 2r_n)) \, d\eta(y) \nonumber\\
&\quad + c^2n^2 \theta_g r_n^{-2\zeta} 
+ n^2 \left( \int f_*\mu(B(f(y), r_n)) \, d\eta(y) \right)^2 + 2ng \int f_*\mu(B(f(y), r_n)) \, d\eta(y).\nonumber
\end{align}

We now proceed to estimate $\mu \otimes \eta \bigl\{ (x, y) : \widehat{m}^f_n(x, y) \ge r_n \bigr\}$
using inequality \eqref{eqmnvarsn}.  
Combining inequalities~\eqref{varestimate} and~\eqref{duas-somas}, and recalling
that 
    \begin{equation*}
\E(Q^f_n)=n\int f_{\ast}\mu\left(B(z,r_n)\right)d(f_{\ast}\eta)(z),
\end{equation*}
 we obtain

\begin{align}\label{rhs}
\frac{\operatorname{Var}(Q_n^f)}{\mathbb{E}(Q_n^f)^2}
&\leq 
\frac{
n^2 L^2 \rho^{-2} r_n^{-2} \theta_g
+ c^2n^2 \theta_g r_n^{-2\zeta}
+ 2n^2 r_n^{-\xi} \rho^{\beta} \int f_*\mu(B(f(y), 2r_n)) \, d\eta(y)
}{
\left(n \int f_*\mu(B(f(y), r_n)) \, d\eta(y)\right)^2
} \nonumber\\
&\quad +\frac{\mathbb{E}(Q_n^f)^2}{\mathbb{E}(Q_n^f)^2}+
\frac{
2n g \int f_*\mu(B(f(y), r_n)) \, d\eta(y)
}{
\left(n \int f_*\mu(B(f(y), r_n)) \, d\eta(y)\right)^2
}-\frac{\mathbb{E}(Q_n^f)^2}{\mathbb{E}(Q_n^f)^2}.
\end{align}

\medskip
Our goal now is to estimate each term on the right-hand side of the inequality above.

We begin with the third term. 
By the definition of \( \overline{C}^{f}_{\mu,\eta} \), we have that for \( n \) large enough,
\begin{equation}\label{rn}
r_n^{\overline{C}^{f}_{\mu,\eta} + \epsilon} \ \leq \ \int f_* \mu\left( B\left( f(y), r_n \right) \right) \, d\eta(y)\leq  \int f_* \mu\left( B\left( f(y), 2r_n \right) \right) \, d\eta(y) \ \leq \ 1.
\end{equation}
Therefore
\begin{eqnarray*}
\frac{2n^2 r_n^{-\xi} \rho^{\beta} \int f_*\mu(B(f(y), 2r_n)) \, d\eta(y)}{
\left(n \int f_*\mu(B(f(y), r_n)) \, d\eta(y)\right)^2}
\leq
\frac{2 r_n^{-\xi} \rho^{\beta}}{r_n^{2(\overline{C}^{f}_{\mu,\eta}+ \epsilon)}}. 
\end{eqnarray*}

Replacing $r_n = e^{-k_n}$, with $k_n$ as defined at the beginning of this proof, we deduce

\begin{eqnarray*}
\frac{2n^2 r_n^{-\xi} \rho^{\beta} \int f_*\mu(B(f(y), 2r_n)) \, d\eta(y)}{
\left(n \int f_*\mu(B(f(y), r_n)) \, d\eta(y)\right)^2}
&\leq& 2 \rho^{\beta} n^{\xi / (\overline{C}^{f}_{\mu,\eta} + \epsilon)} 
(\log n)^{\frac{b \xi}{\overline{C}^{f}_{\mu,\eta} + \epsilon}} 
n^2 (\log n)^{2b}.
\end{eqnarray*}

Now we set $\rho = n^{-\delta}$ and choose
\begin{equation}\label{delta}
\delta = \frac{2}{\beta} + \frac{\xi}{\beta \left( \overline{C}^{f}_{\mu,\eta} + \epsilon \right)}+ \frac{\xi}{ \left( \overline{C}^{f}_{\mu,\eta} + \epsilon \right)}.
\end{equation}
It follows that
\begin{eqnarray}\label{term32}
\frac{2n^2 r_n^{-\xi} \rho^{\beta} \int f_*\mu(B(f(y), 2r_n)) \, d\eta(y)}{
\left(n \int f_*\mu(B(f(y), r_n)) \, d\eta(y)\right)^2}
&\leq&
2 (\log n)^{\, b\left(2 + \frac{\xi}{\overline{C}^{f}_{\mu,\eta} + \epsilon} \right)}.
\end{eqnarray}
Next we analyze the first term on the right-hand side of inequality~\eqref{rhs}. Using again~\eqref{rn}, we obtain
\begin{equation*}
\frac{n^2 L^2 \rho^{-2} r_n^{-2} \theta_g}{\left(n \int f_{\ast} \mu \left( B(f(y), r_n) \right) \, d\eta(y)\right)^2}
\ \leq \
L^2 \rho^{-2} \, r_n^{-2} \theta_g  r_n^{-2(\overline{C}^f_{\mu,\eta} + \varepsilon)}.
\end{equation*}
By substituting the definitions of \( r_n \) and \( k_n \) into the expression above, it follows that
\begin{eqnarray}\label{term1a}
\frac{n^2 L^2 \rho^{-2} r_n^{-2} \theta_g}{
\left(n \int f_{\ast} \mu\!\left( B(f(y), r_n) \right)\, d\eta(y)\right)^2}
&\leq&
L^2 \rho^{-2} \theta_g \,
n^{2} (\log n)^{2b} \,
n^{2/(\overline{C}^f_{\mu,\eta}+\eps)}
(\log n)^{\frac{2 b}{\overline{C}^f_{\mu,\eta}+\eps}}
\nonumber
\end{eqnarray}
Using the definitions \( g = \log(n^{\gamma}) \), \( \theta_g = (a^g)^{\alpha} \),
and \( \rho = n^{-\delta} \), we obtain
\[
\rho^{-2} \theta_g \, n^{2} \, n^{\frac{2}{\overline{C}^f_{\mu,\eta} + \varepsilon}}
= n^{2\delta + 2 + \frac{2}{\overline{C}^f_{\mu,\eta} + \varepsilon}} 
   \cdot \left( a^{\log(n^{\gamma})} \right)^{\alpha}
= n^{2\delta + 2 + \frac{2}{\overline{C}^f_{\mu,\eta} + \varepsilon} + \alpha \gamma \log a}.
\]

Since Hypothesis~(H1) ensures that \(\alpha>0\) and \(\log a < 0\), and recalling
the definition of \(\delta\) given in~\eqref{delta}, choosing
\begin{equation}\label{gamma}
\gamma >
\frac{1}{\alpha \lvert \log a \rvert}
\left(
\frac{4}{\beta}
+ \frac{2\xi}{\beta\bigl(\overline{C}^f_{\mu,\eta}+\varepsilon\bigr)}
+ \frac{2\xi}{\overline{C}^f_{\mu,\eta}+\varepsilon}
+ 2
+ \frac{2}{\overline{C}^f_{\mu,\eta}+\varepsilon}
\right),
\end{equation}
yields
\[
2\delta + 2 + \frac{2}{\overline{C}^f_{\mu,\eta} + \varepsilon}
+ \alpha \gamma \log a < 0,
\]
and we conclude that
\begin{equation}\label{term1}
\frac{n^2 L^2 \rho^{-2} r_n^{-2} \theta_g}{\left( n \int f_{\ast} \mu \left( B(f(y), r_n) \right) \, d\eta(y) \right)^2}
\leq
L^2 (\log n)^{2b \left(1 + \frac{1}{\overline{C}^f_{\mu,\eta} + \varepsilon} \right)}.
\end{equation}
Let us now turn to the second term on the right-hand side of inequality~\eqref{rhs}. 
Using again~\eqref{rn} and the definitions of $r_n$ and $k_n$, we obtain
\begin{eqnarray*}
\frac{c^2 n^2 \theta_g r_n^{-2\zeta}}{\left(n\int f_{\ast} \mu \left( B(f(y), r_n) \right) \, d\eta(y)\right)^2}
&\leq& 
c^2 \theta_g \, 
n^{2}(\log n)^{2b} \,
n^{\frac{2\zeta}{\overline{C}^f_{\mu,\eta}+\varepsilon}}
(\log n)^{\frac{2 b \zeta}{\overline{C}^f_{\mu,\eta}+\varepsilon}}.
\end{eqnarray*}
By our choice of $\gamma$ in~\eqref{gamma}, we have
\[
\gamma > \frac{1}{\alpha |\log a|}
\left( 2 + \frac{2\zeta}{\overline{C}^f_{\mu,\eta} + \varepsilon} \right).
\]
Therefore, we obtain
\begin{equation}\label{term2}
\frac{c^2 n^2 \theta_g r_n^{-2\zeta}}{
\left(n \int f_{\ast} \mu\left(B(f(y), r_n)\right) \, d\eta(y)\right)^2}
\leq
c^2 (\log n)^{2b\left(1 + \frac{\zeta}{\overline{C}^f_{\mu,\eta} + \varepsilon}\right)}.
\end{equation}

For the remaining term, we use inequality~\eqref{rn} and the identity
\[
r_n^{\overline{C}^f_{\mu,\eta} + \varepsilon} = \frac{1}{n(\log n)^b},
\]
to get
\begin{equation}\label{term4a}
\frac{2n g \int f_{*} \mu\!\left(B(f(y), r_n)\right) \, d\eta(y)}{
\left(n \int f_{*} \mu\!\left(B(f(y), r_n)\right) \, d\eta(y)\right)^2}
\leq
\frac{2g}{n\, r_n^{\overline{C}^f_{\mu,\eta} + \varepsilon}}\leq 2\gamma (\log n)^{1 + b}.
\end{equation}

Taking $b < -2$, and substituting \eqref{term32}, \eqref{term1}, \eqref{term2}, and \eqref{term4a} into \eqref{rhs}, and then recalling  inequality \eqref{eqmnvarsn}, we obtain

\[\mu \otimes \eta \left((x, y) : \widehat{m}^f_n(x, y)\geq r_n\right)\leq\frac{\var(Q^f_n)}{\E(Q^f_n)^2}\leq \mathcal{O}((\log n)^{-1}).\]
Following the same steps as in the proof of Theorem~2, 
and using the same subsequence $n_\ell=\lceil e^{\ell^2}\rceil$ together with the Borel-Cantelli argument, we conclude
\[
\underset{\ell \rightarrow \infty}{\underline{\lim}}
\frac{\log \widehat{m}_{n_{\ell}}^f(x,y)}{-\log n_{\ell}}
\ge 
\frac{1}{\overline{C}^f_{\mu,\eta}+\epsilon}.
\]

Since
\[
\underset{n\rightarrow+\infty}{\underline\lim}
\frac{\log \widehat{m}^f_n(x,y)}{-\log n}
=
\underset{\ell \rightarrow+\infty}{\underline\lim}
\frac{\log \widehat{m}^f_{n_\ell}(x,y)}{-\log n_\ell},
\]
the theorem follows, as $\epsilon$ may be chosen arbitrarily small.

\end{proof}


\section{Shortest distance between orbits for random dynamical systems}\label{sectionshortestrandom}

Let $X \subset \mathbb{R}^n$ and let $(\Omega, B(\Omega), \mathbb{P}, \vartheta)$ be a probability measure preserving system, 
where $\Omega$ is a metric space and $B(\Omega)$ its Borel $\sigma$-algebra. 
We now introduce the notion of a random dynamical system acting on $X$ over the base $(\Omega, B(\Omega), \mathbb{P}, \vartheta)$.

\begin{definition}
A random dynamical system \(\textbf{T} = (T_{\omega})_{\omega \in \Omega}\) on \(X\) over \((\Omega, B(\Omega), \mathbb{P}, \vartheta)\) is generated by maps \(T_\omega\) such that \((\omega, x) \in \Omega \times X \mapsto T_\omega(x)\) is measurable. The map \(\mathcal{T}: \Omega \times X \to \Omega \times X\) defined by \(\mathcal{T}(\omega, x) = (\vartheta(\omega), T_\omega(x))\) is the dynamics of the random dynamical system generated by \(\textbf{T}\) and is called the \emph{skew-product}.
\end{definition}

We are particularly interested in invariant measures for the skew-product dynamics.

\begin{definition}
A probability measure \(\mu\) on the product space \(\Omega \times X\) is said to be an \emph{invariant measure} for the random dynamical system \(\textbf{T}\) if it satisfies:
\begin{itemize}
\item[1.] \(\mu\) is \(\mathcal{T}\)-invariant,
\item[2.] \(\pi_* \mu = \mathbb{P}\),
\end{itemize}
where \(\pi: \Omega \times X \to \Omega\) is the canonical projection.
\end{definition}
Visually, this relationship can be represented by the following commutative diagram:
\[
\begin{tikzcd}
(\omega, x) \arrow[r, "\mathcal{T}"] \arrow[d, "\pi"'] 
  & (\vartheta(\omega), T_\omega(x)) \arrow[d, "\pi"] \\[4pt]
\omega \arrow[r, "\vartheta"'] 
  & \vartheta(\omega)
\end{tikzcd}
\]

In this setting, the probability measure \(\mu\) can be disintegrated over \(\mathbb{P}\) (see the Regular Conditional Probability Theorem, \cite{kiferliu}): there exists a family of probability measures \(\mu_\omega\), depending measurably on \(\omega\), such that for any bounded function \(f\) on \(X\), the following holds:
\[
\int f \, d\mu = \int \left( \int f(x) \, d\mu_\omega(x) \right) d\mathbb{P}(\omega).
\]
Let \(\bar{\mu}(\cdot) = \int \mu_\omega(\cdot) \, d\mathbb{P}(\omega)\) denote the second marginal of \(\mu\).


We now introduce a second family of random maps \(\textbf{S} = (S_{\omega})_{\omega \in \Omega}\), defined on the same space \(X\) over \((\Omega, B(\Omega), \mathbb{P}, \tilde{\vartheta})\). The associated skew-product is \(\mathcal{S}: \Omega \times X \to \Omega \times X\), defined by
\[
\mathcal{S}(\omega, x) = (\tilde{\vartheta}(\omega), S_{\omega}(x)),
\]
and the {invariant measure} for the random dynamical system \(\textbf{S}\) is denoted by \(\eta\). With these notations, let \(\bar{\eta}(\cdot) = \int \eta_\omega(\cdot) \, d\mathbb{P}(\omega)\) denote the second marginal of \(\eta\).

\begin{definition}
We define the \emph{shortest distance between two different random orbits} as

\[
\widehat{m}_n^{\omega, \tilde{\omega}}(x,y) = \min_{i = 0, \dots, n-1} \, d\left(T_\omega^i(x), \, S_{\tilde{\omega}}^i(y)\right),
\]
where, for all \(\omega \in \Omega\), the iterates of \(T_\omega\) are the random composition of  \(T_\omega\)'s defined as:
\[
T_\omega^0 = \mathrm{Id}, \quad
T_\omega^1 = T_\omega, \quad
T_\omega^2 = T_{\vartheta(\omega)} \circ T_\omega,
\]
and more generally,
\[
T_\omega^n = T_{\vartheta^{n-1}(\omega)} \circ \cdots \circ T_{\vartheta(\omega)} \circ T_\omega.
\]
Similarly, \(S_{\tilde{\omega}}^n\) are the random compositions of the maps \(S_{\tilde{\omega}}\).
\end{definition}

To obtain a result that links the shortest distance between two different orbits and random dynamical systems we need to assume a hypothesis on the invariant measure and on the (annealed) decay of correlations of the random dynamical system.
Namely,
\begin{description}
  \item[(a)] There exist $r_0>0, \ \xi \geq 0$ and $\beta\geq 1$ such that for almost every $y \in X$ and any $r_0>r> \rho>0,$
$$\bar{\mu}(B(y,r+\rho)\backslash B(y,r-\rho)) \leq r^{-\xi}\rho^{\beta},$$
where  $\bar{\mu}$ is the second marginal of \(\mu\).

\item[(b)] (Annealed decay of correlations) $\forall n \in \mathbb{N}^*$, $\psi,\ \phi\in Lip(X,\R)$ we have,
$$\left|\int_{\Omega \times X} \psi (T^n_\omega(x))\phi(x) \ d\mu(\omega,x) - \int_{X} \psi \ d\bar{\mu} \int_{X} \phi \ d\bar{\mu} \right| \leq \|\psi \|_{Lip}\|\phi \|_{Lip} \theta_n$$
and
$$\left|\int_{\Omega \times X} \psi (S^n_{\tilde{\omega}}(x))\phi(x) \ d\eta(\tilde{\omega},x) - \int_{X} \psi \ d\bar{\eta} \int_{X} \phi \ d\bar{\eta} \right| \leq \|\psi \|_{Lip}\|\phi \|_{Lip} \theta_n$$
with $\theta_n =e^{-n}$.
\end{description}
The following theorem establishes the asymptotic behavior of the minimal distance between two random orbits when the random dynamical systems satisfy the regularity and decay conditions stated above.

\begin{theorem}\label{theoremrandom}
Let $\textbf{T}$ and $\textbf{S}$ be two random dynamical systems on a common space $X$, defined over $(\Omega, B(\Omega), \mathbb{P}, \vartheta)$ and $(\Omega, B(\Omega), \mathbb{P}, \tilde{\vartheta})$, respectively. Let $\mu$ and $\eta$ be invariant measures under the respective random dynamical systems $\textbf{T}$ and $\textbf{S}$ such that $\underline{C}_{\bar{\mu},\bar{\eta}}>0$. Then for $\mu \otimes \eta$-almost every $(\omega, x,\tilde{\omega},\tilde{x}) \in \Omega \times X \times \Omega \times X,$ if the random dynamical system satisfies assumptions $(a)$ and $(b)$, and if $C_{\bar{\mu},\bar{\eta}}$ exists, then
$$  \underset{n \rightarrow \infty}{\lim}\frac{\log \widehat{m}_n^{\omega, \tilde{\omega}}(x,\tilde{x})}{-\log n} = \frac{1}{C_{\bar{\mu},\bar{\eta}}}.$$

\end{theorem}

\begin{proof}
We follow the approach proposed in \cite{Rousseau}. 
The proof of this theorem relies on Theorems~\ref{thineq} and~\ref{thprinc}, 
applied to the dynamical systems 
\((\Omega \times X, B(\Omega \times X), \mu, \mathcal{T})\) and 
\((\Omega \times X, B(\Omega \times X), \eta, \mathcal{S})\), 
together with the observation function \(f\) defined as
\[
f : \Omega \times X \to X, \quad (\omega, x) \mapsto x.
\]

With this observation, for all $z, t \in \Omega \times X$, we can relate the shortest distance between two different observed orbits to the shortest distance between two different random orbits.  
Let $z = (\omega, x)$ and $t = (\tilde{\omega}, \tilde{x})$. Then,
\begin{eqnarray*}
\widehat{m}_n^f(z,t) 
&=& \min_{i = 0, \ldots, n-1} 
\left( d\left( f\left(\mathcal{T}^i(\omega, x)\right), f\left(\mathcal{S}^i(\tilde{\omega}, \tilde{x})\right) \right) \right) \\
&=& \min_{i = 0, \ldots, n-1} 
\left( d\left( T_\omega^i(x), S_{\tilde{\omega}}^i(\tilde{x}) \right) \right) \\
&=& \widehat{m}_n^{\omega, \tilde{\omega}}(x, \tilde{x}).
\end{eqnarray*}

Moreover, we can identify the pushforward measures as 
$f_*\mu = \bar{\mu}$ and $f_*\eta = \bar{\eta}$.  
Therefore, in view of the definition of the symmetric Rényi divergence, we obtain
\[
{C}_{\mu,\eta}^f = {C}_{\bar{\mu}, \bar{\eta}}.
\]
This completes the proof.

\end{proof}


We will present an example of a random dynamical system for which we can apply the previous statement.

\subsection*{Example: Non-i.i.d. random dynamical system}

Let $\Omega = [0,1]$ and $X = \mathbb{T}^1$ denote the one-dimensional torus, both endowed with the Lebesgue measure.

We consider two measurable maps 
\[
\theta: \Omega \to \Omega 
\quad \text{and} \quad 
\tilde{\theta}: \Omega \to \Omega
\]
defined by
\[
\theta(\omega)=
\begin{cases}
2\omega, & \text{if } \omega \in [0,1/2),\\[0.3em]
2\omega - 1, & \text{if } \omega \in [1/2,1],
\end{cases}
\qquad
\tilde{\theta}(\tilde{\omega})=
\begin{cases}
3\tilde{\omega}, & \text{if } \tilde{\omega} \in [0,1/3),\\[0.3em]
3\tilde{\omega} - 1, & \text{if } \tilde{\omega} \in [1/3,2/3),\\[0.3em]
3\tilde{\omega} - 2, & \text{if } \tilde{\omega} \in [2/3,1].
\end{cases}
\]

For each $\omega, \tilde{\omega} \in \Omega$, we define the linear maps on $X$
\[
T_\omega: X \to X, \quad x \mapsto 2x, 
\qquad 
S_{\tilde{\omega}}: X \to X, \quad x \mapsto 3x,
\]
both preserving the Lebesgue measure on $\mathbb{T}^1$.

The corresponding random dynamical systems are described by the skew products
\begin{eqnarray*}
\mathcal{T}: \Omega \times X &\to& \Omega \times X, \\
(\omega,x) &\mapsto& (\theta(\omega), T_\omega(x)),
\end{eqnarray*}
and
\begin{eqnarray*}
\mathcal{S}: \Omega \times X &\to& \Omega \times X, \\
(\tilde{\omega},x) &\mapsto& (\tilde{\theta}(\tilde{\omega}), S_{\tilde{\omega}}(x)).
\end{eqnarray*}

Both $\mathcal{T}$ and $\mathcal{S}$ are $\mathrm{Leb} \otimes \mathrm{Leb}$-invariant, 
where $\mathrm{Leb}$ denotes the Lebesgue measure.
 
It is straightforward to verify that the Lebesgue measure satisfies condition (a).  
Moreover, by \cite{Baladi}, the skew products $\mathcal{T}$ and $\mathcal{S}$ exhibit exponential decay of correlations.  
Since in this example $C_{\bar{\mu}, \bar{\eta}} = 1$ with $\bar{\mu} = \bar{\eta}=\mathrm{Leb}$, Theorem~\ref{theoremrandom} implies that for $\mathrm{Leb}^{\otimes 4}$-almost every 
$(\omega, x, \tilde{\omega}, \tilde{x}) \in [0,1] \times \mathbb{T}^1 \times [0,1] \times \mathbb{T}^1$,
\[
\lim_{n \to \infty} 
\frac{\log \widehat{m}_n^{\omega, \tilde{\omega}}(x,\tilde{x})}{-\log n}
= 1.
\]

\medskip \subsection*{Acknowledgements}The authors would like to thank Jérôme Rousseau for his helpful comments on an earlier version of this manuscript.

\end{document}